\documentclass[11pt,a4]{article}

\usepackage[utf8]{inputenc}
\usepackage[english,main=russian]{babel}

\usepackage{amsfonts, amsmath, amsthm}
\usepackage{amssymb}
\usepackage{url}
\def\change#1{{\color{black}#1}}

\usepackage[
sorting=none
]{biblatex}
\usepackage{mathtools}
\usepackage{makecell}
\usepackage{tikz}
\usepackage{dsfont}
\usepackage{secdot}

\sectiondot{section}
\sectiondot{subsection}
\theoremstyle{definition}

\newcommand{\norm}[1]{\left\lVert#1\right\rVert}
\newtheorem{theorem}{Теорема}
\newtheorem{lemma}{Лемма}
\newtheorem{corollary}{Следствие}
\newtheorem{definition}{Определение}
\newtheorem{remark}{Замечание}

\newcommand\myeq{\mathrel{\stackrel{\makebox[0pt]{\mbox{\normalfont\tiny def}}}{=}}}
\newcommand{\leqarg}[1]{\ensuremath{\stackrel{\text{#1}}{\leq}}}
\newcommand{\eqarg}[1]{\ensuremath{\stackrel{\text{#1}}{=}}}

\newcommand\uprule{\rule{0mm}{1.9ex}}
\DeclareMathOperator*{\argmin}{arg\,min}
\DeclareMathOperator*{\Argmin}{Arg\,min}

\newcommand*\circled[1]{\tikz[baseline=(char.base)]{
		\node[shape=circle,draw,inner sep=2pt] (char) {#1};}}

\allowdisplaybreaks

\begin{document}

\textbf{УДК} 519.85

\begin{center}
\textbf{Быстрый градиентный спуск для задач выпуклой минимизации с оракулом, выдающим $(\delta,L)$-модель функции в запрошенной точке\footnote{Исследование А.И. Тюрина финансировалось в рамках государственной поддержки ведущих университетов Российской Федерации «5-100». Работа А.В. Гасникова по основному материалу статьи была поддержана грантом РФФИ 18-31-20005 мол\_а\_вед. Работа А.В. Гасникова в Приложении была поддержана грантом РНФ 17-11-01027.}}

{\bf
\copyright\,2019 г.\,\,
А. В. Гасников$^{*, **, ***}$,
А. И. Тюрин$^*$
}

(*101000 Москва, ул. Мясницкая, 20, НИУ ВШЭ;

**141700 Долгопрудный, М.о., Институтский пер., 9, НИУ МФТИ;

***127051 Москва, Бол. Каретный пер., 19, стр. 1, Ин-т пробл. передачи информац. РАН;)
 
e-mail: atyurin@hse.ru

Поступила в редакцию: 08.11.2017 г.
\end{center}

\renewcommand{\abstractname}{\vspace{-\baselineskip}}
\begin{abstract}
Предлагается новая концепция $(\delta,L)$-модели функции, которая обобщает концепцию $(\delta,L)$-оракула Деволдера--Глинера--Нестерова. В рамках этой концепции строятся градиентный спуск, быстрый градиентный спуск и показывается, что многие известные ранее конструкции методов (композитные методы, методы уровней, метод условных градиентов, проксимальные методы) являются частными случаями предложенных в данной работе методов. Библ. 34.
\end{abstract}

\textbf{Ключевые слова:} градиентный спуск, быстрый градиентный спуск, модель функции, универсальный метод, метод условного градиента, композитная оптимизация.

\section{Введение}
Градиентный спуск и быстрый градиентный спуск являются, пожалуй, одними из самых популярных сейчас численными методами оптимизации. В основу обоих методов положена идея аппроксимации функции в исходной точке (текущем положении метода) мажорирующим ее параболоидом вращения и выбора точки минимума параболоида вращения в качестве нового положения метода. Таким образом, реализуется принцип ``разделяй и властвуй'', исходно сложная задача разбивается на совокупность более простых задач. Методы второго порядка, квазиньютоновские методы, методы с кубической регуляризацией, чебышевские методы, композитные методы и методы уровней наводят на мысль, что совсем не обязательно использовать в описанном выше подходе именно параболоиды вращения. Можно использовать более сложные функции, которые строят более точную локальную модель функции в рассматриваемой точке, что приводит в итоге к более быстрой сходимости метода. Здесь можно выделить два направления: одно направление связано с использованием старших производных в модели функции, второе направление связано с занесением в модель части постановки задачи, например, в случае когда функция представляется в виде суммы двух функций: одну из них можно заменять параболоидом вращения в модели, а вторую оставить как есть. Второе направление является более новым и на данный момент по сути и не является направлением. Насколько нам известно, до настоящего момента имеющиеся здесь разрозненные результаты не было попыток связать. В данной работе предпринимается такая попытка. Стоит отметить, что в случае с параболоидом вращения решение вспомогательной задачи, как правило не представляет большого труда и часто может быть сделано по явным формулам, таким образом, без ошибок. Совсем другое дело, когда мы рассматриваем второй подход, в котором типична обратная ситуация -- вспомогательную задачу можно решить только приближенно. В этой связи в данной работе прорабатывается возможность неточного решения вспомогательной задачи.

Общая задача оптимизации, когда имеется только информация о липшицевости градиента или о липшицевости функции, хорошо изучена и для нее имеются нижние и верхние оценки \cite{nesterov2010introductory}, \cite{nemirovskiy1979slognost}. В последнее время стала популярна структурная оптимизация, когда мы имеем априорную информацию о структуре задачи. Дополнительная информация о задаче позволяет находить новые методы и улучшать верхние оценки. В частности, композитная постановка задачи, когда заданная функция имеет вид суммы гладкой и негладкой функции \cite{nesterov2013gradient}, может быть решена со скоростью быстрого градиентного метода, используя дополнительную информацию о негладком слагаемом в сумме.   Помимо этого, часто можно находить минимум негладкой функций со скоростью быстрого градиентного метода, несмотря на то, что в общем случае это невозможно для негладких задач. Как пример, можем рассмотреть минимаксную задачу, которая является негладкой, но для нее предложен быстрый градиентный метод \cite{nesterov2010introductory}. Основной целью данной работы является попытка объединить различные подходы в один, таким образом, предложить метод, который бы смог в себя включать ранее предложенные концепции. В разд. 4 мы привели большое количество примеров задач оптимизации, которые могут быть решены нашим методом. Для этого вводится понятие $(\delta,L)$-модели, которое по своей сути является обобщением определения липшицевости градиента или концепции $(\delta,L)$-оракула Деволдера--Глинера--Нестерова.\\

Статья построена следующим образом. В разд. 2 на базе концепции $(\delta,L)$-оракула \cite{devolder2013first} приводится оригинальная концепция $(\delta,L)$-модели функции, при этом дополнительно допускается возможность неточного решения вспомогательной задачи \cite{ben-tal2015lectures}. Также в этом разделе приводится обобщение метода градиентного спуска на случай работы с введенной моделью. В разд. 3 результаты разд. 2 с градиентного спуска переносятся на быстрый градиентный спуск. В разд. 4 содержатся приложения предложенных в разд. 2, 3 методов в концепции $(\delta,L)$-модели к различным постановкам задач. Именно в этом разделе показывается, что предложенная концепция и методы, действительно, позволяют собрать имеющиеся в данном направлении результаты.
В Приложении приводится обоснование того, что выбранная, следуя А.С. Немировскому, концепция точности, с которой решается вспомогательная задача вполне практична и для ограниченных гладких постановок задач сводится к обычной концепции сходимости по функции (стоит отметить, что есть и другие концепции \cite{taylor2015exact}, \cite{ochs2017non}, \cite{mairal2013optimization}).

\section{Градиентный метод с оракулом, использующим $(\delta, L)$-модель}
\label{gradMethod}
Опишем сначала общую постановку задачи выпуклой оптимизации \cite{nesterov2010introductory}. 
Пусть определена функция $F(x): Q \longrightarrow \mathds{R}$ и дана произвольная норма $\norm{}$ в $\mathds{R}^n$. Сопряженная норма определяется следующим образом:
\begin{gather}\norm{\lambda}_* = \max\limits_{\norm{\nu} \leq 1;\nu \in \mathds{R}^n}\langle \lambda,\nu\rangle\,\,\,\,\forall \lambda \in \mathds{R}^n.\end{gather}
Будем полагать, что:
\begin{enumerate}
	\item $Q \subseteq \mathds{R}^n$, выпуклое, замкнутое.
	\item $F(x)$ -- непрерывная и выпуклая функция на $Q$.
	\item $F(x)$ ограничена снизу на $Q$ и достигает своего минимума в некоторой точке (необязательно единственной) $x_* \in Q$.
\end{enumerate}

Рассмотрим следующую задачу оптимизации:
\begin{align}
\label{mainTask3}
F(x) \rightarrow \min_{x \in Q}.
\end{align}

Введем два понятия: прокс-функция и дивергенция Брэгмана \cite{gupta2008bregman}.
\begin{definition}
$d(x):Q \rightarrow \mathds{R}$ называется прокс-функцией, если $d(x)$ непрерывно дифференцируемая на $\textnormal{int }Q$ и $d(x)$ является 1-сильно выпуклой относительно нормы $\norm{}$ на $\textnormal{int }Q$.
\end{definition}
\begin{definition}
Дивергенцией Брэгмана называется 
\begin{align}
V(x,y) \myeq d(x) - d(y) - \langle\nabla d(y), x - y\rangle,
\end{align}
где $d(x)$ -- произвольная прокс-функция.
\end{definition}
Легко показать \cite{ben-tal2015lectures}, что \begin{gather*}V(x,y) \geq \frac{1}{2}\norm{x - y}^2.\end{gather*}

Далее мы введем определение $(\delta, L)$-модели функции \cite{gasnikov2017universal}, которая является прямым обобщением $(\delta, L)$-оракула \cite{devolder2013first}, \cite{devolder2014first}, \cite{devolder2013exactness}.
\begin{definition}
\leavevmode
\label{gen_delta_L_oracle}
Будем говорить, что имеем $(\delta, L)$-модель функции $F(x)$ в точке $y$, и обозначать эту модель $(F_{\delta}(y), \psi_{\delta}(x,y))$, если для любого $x \in Q$ справедливо неравенство
\begin{gather}
\label{exitLDLOrig}
0 \leq F(x) - F_{\delta}(y) - \psi_{\delta}(x,y) \leq \frac{L}{2}\norm{x - y}^2 + \delta,
\end{gather}
причем
\begin{gather}
\label{exitLDLOrig_eqaul}
\psi_{\delta}(x,x) = 0 \,\,\,\, \forall x \in Q
\end{gather}
и $\psi_{\delta}(x,y)$ -- выпуклая функция по $x$ для $\forall y \in Q$.
\end{definition}
Будем считать, что для $F(x)$ имеются такие $\delta$ и $L$, что существует $(\delta, L)$-модель в любой точке $x \in Q$. Соответствующие примеры представлены в разд. \ref{sledviya}.

\begin{corollary}

Возьмем $x = y$ в (\ref{exitLDLOrig}) и воспользуемся (\ref{exitLDLOrig_eqaul}), тогда
\begin{gather}
\label{exitLDLOrig2}
F_{\delta}(y) \leq F(y) \leq F_{\delta}(y) + \delta \,\,\,\,\forall y \in Q.
\end{gather}
\end{corollary}

Рассмотрим концепцию неточного решения задачи, представленную в \cite{ben-tal2015lectures}.
\begin{definition}
\label{solNemirovskiy}
Пусть имеется задача
\begin{gather*}
\psi(x) \rightarrow \min_{x \in Q},
\end{gather*}
где $\psi(x)$ -- выпуклая, тогда $\Argmin_{x \in Q}^{\widetilde{\delta}}\psi(x)$ -- множество таких $\widetilde{x}$, что
\begin{gather*}
\exists h \in \partial\psi(\widetilde{x}), \,\,\, \langle h, x - \widetilde{x}  \rangle \geq -\widetilde{\delta} \,\,\,\, \forall x \in Q.
\end{gather*}
Произвольный элемент из $\Argmin_{x \in Q}^{\widetilde{\delta}}\psi(x)$ будем обозначать через $\argmin_{x \in Q}^{\widetilde{\delta}}\psi(x)$.

\end{definition}

Рассмотрим простое следствие. Пусть $\widetilde{x} \in \Argmin_{x \in Q}^{\widetilde{\delta}}\psi(x)$, тогда из выпуклости верно: $\psi(x) \geq \psi(\widetilde{x}) + \langle h, x - \widetilde{x}  \rangle \geq \psi(\widetilde{x})-\widetilde{\delta}$, где $h \in \partial\psi(\widetilde{x})$. Возьмем $x = x_*$, тогда $\psi(\widetilde{x}) - \psi(x_*) \leq \widetilde{\delta}$. Таким образом, из $\widetilde{x} \in \Argmin_{x \in Q}^{\widetilde{\delta}}\psi(x)$ следует, что $\widetilde{x}$ является $\widetilde{\delta}$-решением по функции. Обратное вообще говоря неверно, в доказательстве далее ключевым образом будет использоваться именно условие на $\widetilde{\delta}$-решение из определения \ref{solNemirovskiy}, которое является более строгим.

Рассмотрим обобщение алгоритма градиетного спуска для задачи (\ref{mainTask3}) (см. \cite{gasnikov2017universal}). В данном алгоритме будем предполагать, что дана начальная точка $x_0$, $N$ -- количество шагов метода, $L_0$ -- константа, которая имеет смысл предположительной "локальной"\, константы Липшица градиента в точке $x_0$. Также на вход алгоритму подаются последовательности $\{\delta_k\}_{k=0}^{N-1}$, $\{\widetilde{\delta}_k\}_{k=0}^{N-1}$. Должно быть выполнено, что для любого $k$ существует $(\delta_k, L)$-модель для $F(x)$ в любой точке $x \in Q$. Последовательность $\{\widetilde{\delta}_k\}_{k=0}^{N-1}$ -- точности решения из определения \ref{solNemirovskiy}, причем в зависимости от задачи они могут быть равными нулю, иметь постоянное значение или меняться от итерации к итерации.

Стоит отметить, что далее нигде явно в методе не будет использоваться константа $L$. Будем считать, что $L_0 \leq L$, иначе во всех оценках далее следует полагать $L := \max\left(L_0, L\right)$. 

Представим алгоритм градиентного метода с оракулом, использующим $(\delta, L)$-модель.\\

\textbf{Алгоритм:}

\textbf{Дано:} $x_0$ -- начальная точка, $N$ -- количество шагов, $\{\delta_k\}_{k=0}^{N-1}$, $\{\widetilde{\delta}_k\}_{k=0}^{N-1}$ -- последовательности и $L_0 > 0$.

\textbf{0 - шаг:}
\begin{gather*}
L_1 := \frac{L_0}{2}.
\end{gather*}

\textbf{$\boldsymbol{k+1}$ - шаг:}
\begin{gather*}
\alpha_{k+1} := \frac{1}{L_{k+1}},
\end{gather*}
\begin{equation}
\begin{gathered}
\phi_{k+1}(x) = V(x, x_k) + \alpha_{k+1}\psi_{\delta_k}(x, x_{k}),\\
x_{k+1} := {\argmin_{x \in Q}}^{\widetilde{\delta}_k}\phi_{k+1}(x). \label{equmir2DL_G}
\end{gathered}
\end{equation}
Если выполнено условие
\begin{equation}
\begin{gathered}
F_{\delta_k}(x_{k+1}) \leq F_{\delta_k}(x_{k}) + \psi_{\delta_k}(x_{k+1}, x_{k}) +\\  + \frac{L_{k+1}}{2}\norm{x_{k+1} - x_{k}}^2 + \delta_k,
\label{exitLDL_G}
\end{gathered}
\end{equation}
то
\begin{gather*}
L_{k+2} := \frac{L_{k+1}}{2}
\end{gather*}
и перейти к следующему шагу, иначе
\begin{gather*}
L_{k+1} := 2L_{k+1}
\end{gather*}
и повторить текущий шаг.

\begin{remark}
\leavevmode
\label{remark_maxmin}
Для всех $k \geq 0$ выполнено
\begin{equation*}
	L_{k} \leq 2L.
\end{equation*}
Для $k = 0$ верно из того, что $L_0 \leq L$. Для $k \geq 1$ это следует из того, что мы выйдем из внутреннего цикла, где подбирается $L_k$, ранее, чем $L_{k}$ станет больше $2L$. Выход из цикла гарантируется тем, что по условию существует $(\delta_k, L)$-модель для $F(x)$ в любой точке $x \in Q$.
\end{remark}

Докажем важную лемму, которая нам пригодится далее.

\begin{lemma}
	Пусть $\psi(x)$ выпуклая функция и 
	\begin{gather*}
	y = {\argmin_{x \in Q}}^{\widetilde{\delta}}\{\psi(x) + V(x,z)\}.
	\end{gather*}
	Тогда выполнено неравенство
	\begin{equation*}
	\psi(x) + V(x,z) \geq \psi(y) + V(y,z) + V(x,y) - \widetilde{\delta} \,\,\,\, \forall x \in Q.
	\end{equation*}
	\label{lemma_maxmin_2}
\end{lemma}
\begin{proof}\renewcommand{\qedsymbol}{}
	По определению \ref{solNemirovskiy}:
	\begin{gather*}
		\exists g \in \partial\psi(y), \,\,\, \langle g + \nabla_y V(y, z), x - y  \rangle \geq -\widetilde{\delta} \,\,\,\, \forall x \in Q.
	\end{gather*}
	Тогда неравенство
	\begin{gather*}
		\psi(x) - \psi(y) \geq \langle g, x - y  \rangle \geq \langle \nabla_y V(y, z), y - x  \rangle - \widetilde{\delta}
	\end{gather*}
и равенство
	\begin{gather*}
	\langle \nabla_y V(y, z), y - x  \rangle = \langle \nabla d(y) - \nabla d(z), y - x  \rangle = d(y) - d(z) - \langle \nabla d(z), y - z  \rangle +\\ + d(x) - d(y) - \langle \nabla d(y), x - y  \rangle - d(x) + d(z) + \langle \nabla d(z), x - z  \rangle = \\=
	V(y,z) + V(x,y) - V(x,z)
	\end{gather*}
завершают доказательство.
	
\end{proof}

\begin{lemma}
\label{lemma_maxmin_3DL_G}
	Для любого $x \in Q$ выполнено неравенство
	\begin{gather*}
		\change{\alpha_{k+1}F(x_{k+1}) - \alpha_{k+1}F(x)\leq V(x, x_k) - V(x, x_{k+1}) + \widetilde{\delta}_k + 2\delta_k\alpha_{k+1}}
	\end{gather*}
\end{lemma}
\begin{proof}\renewcommand{\qedsymbol}{}
	Рассмотрим цепочку неравенств:
	\begin{gather*}
	F(x_{k+1}) \leqarg{(\ref{exitLDL_G}), (\ref{exitLDLOrig2})} F_{\delta_k}(x_{k}) + \psi_{\delta_k}(x_{k+1}, x_{k})  + \frac{L_{k+1}}{2}\norm{x_{k+1} - x_{k}}^2 + 2\delta_k \leq \\ \leq
	F_{\delta_k}(x_{k}) + \psi_{\delta_k}(x_{k+1}, x_{k})  + \frac{1}{\alpha_{k+1}}V(x_{k+1}, x_{k}) + 2\delta_k
	\leq_{{\tiny \circled{1}}} \\\leq
	 F_{\delta_k}(x_{k}) + \psi_{\delta_k}(x,x_{k})
	 	 + \frac{1}{\alpha_{k+1}}V(x, x_k) - \frac{1}{\alpha_{k+1}}V(x, x_{k+1}) + \frac{\widetilde{\delta}_k}{\alpha_{k+1}} + 2\delta_k \leqarg{(\ref{exitLDLOrig})} \\ \leq
	 \change{F(x) + \frac{1}{\alpha_{k+1}}V(x, x_k) - \frac{1}{\alpha_{k+1}}V(x, x_{k+1}) + \frac{\widetilde{\delta}_k}{\alpha_{k+1}}+ 2\delta_k.}
	\end{gather*}

\change{Неравенство {\small \circled{1}} следует из леммы \ref{lemma_maxmin_2} с 
$\psi(x) = \alpha_{k+1} \psi_{\delta_k}(x, x_{k})$ и левой части (\ref{exitLDLOrig}).}
\end{proof}

\begin{theorem}
	\label{mainTheoremDL_G}
	\change{Пусть $V(x_*, x_0) \leq R^2$, где $x_0$ -- начальная точка, а $x_*$ -- ближайшая точка минимума к точке $x_0$ в смысле дивергенции Брэгмана, и \begin{gather*}\bar{x}_N= \frac{1}{A_N}\sum_{k=0}^{N-1}\alpha_{k+1} x_{k+1}.\end{gather*} Для предложенного алгоритма выполнено следующее неравенство:
	\begin{align*}
	F(\bar{x}_N) - F(x_*) \leq \frac{2LR^2}{N} + \frac{2L}{N}\sum_{k=0}^{N-1}\widetilde{\delta}_k + \frac{2}{A_N}\sum_{k=0}^{N-1}\alpha_{k+1}\delta_k.
	\end{align*}}
\end{theorem}
\begin{proof}\renewcommand{\qedsymbol}{}
	
	Суммируя неравенство из леммы \ref{lemma_maxmin_3DL_G} по $k = 0, ..., N - 1$, получим:
	\begin{gather*}
	\change{\sum_{k=0}^{N-1}\alpha_{k+1}F(x_{k+1}) - A_N F(x) \leq V(x, x_0) - V(x, x_N) + \sum_{k=0}^{N-1}\widetilde{\delta}_k + 2\sum_{k=0}^{N-1}\alpha_{k+1}\delta_k.}
	\end{gather*}
	Если взять $x = x_*$, то будем иметь
	\begin{gather*}
	\change{\sum_{k=0}^{N-1}\alpha_{k+1}F(x_{k+1}) - A_N F(x_*) \leq R^2 - V(x_*, x_N) + \sum_{k=0}^{N-1}\widetilde{\delta}_k + 2\sum_{k=0}^{N-1}\alpha_{k+1}\delta_k.}
	\end{gather*}
	Используя условие $V(x_*, x_N) \geq 0$, получим
	\begin{gather*}
	\change{\sum_{k=0}^{N-1}\alpha_{k+1}F(x_{k+1}) - A_N F(x_*) \leq R^2 + \sum_{k=0}^{N-1}\widetilde{\delta}_k + 2\sum_{k=0}^{N-1}\alpha_{k+1}\delta_k.}
	\end{gather*}
	\change{Разделим обе части неравенства на $A_N$:}
	\begin{gather*}
	\change{\frac{1}{A_N}\sum_{k=0}^{N-1}\alpha_{k+1}F(x_{k+1}) - F(x_*) \leq \frac{R^2}{A_N} + \frac{1}{A_N}\sum_{k=0}^{N-1}\widetilde{\delta}_k + \frac{2}{A_N}\sum_{k=0}^{N-1}\alpha_{k+1}\delta_k.}
	\end{gather*}
	Если воспользоваться выпуклостью $F(x)$, то будет верно
	\change{\begin{align*}
		F(\bar{x}_N) - F(x_*) &\leq \frac{R^2}{A_N} + \frac{1}{A_N}\sum_{k=0}^{N-1}\widetilde{\delta}_k + \frac{2}{A_N}\sum_{k=0}^{N-1}\alpha_{k+1}\delta_k\\
		&\leq_{{\tiny \circled{1}}} \frac{2LR^2}{N} + \frac{2L}{N}\sum_{k=0}^{N-1}\widetilde{\delta}_k + \frac{2}{A_N}\sum_{k=0}^{N-1}\alpha_{k+1}\delta_k.
	\end{align*}}
	\change{Неравенство {\small \circled{1}} следует из замечания \ref{remark_maxmin} и $\alpha_{k+1} = 1/L_{k+1}$.}
\end{proof}

\section{Быстрый градиентный метод с оракулом, использующим $(\delta, L)$-модель}
\label{fastGradMethod}
Рассмотрим быстрый вариант алгоритма из разд. \ref{gradMethod}.\\

\textbf{Алгоритм:}

\textbf{Дано:} $x_0$ -- начальная точка, $N$ -- количество шагов, $\{\delta_k\}_{k=0}^{N-1}$, $\{\widetilde{\delta}_k\}_{k=0}^{N-1}$ -- последовательности и $L_0 > 0$.

\textbf{0 - шаг:} $y_0 := x_0,\,
u_0 := x_0,\,
L_1 := \frac{L_0}{2},\,
\alpha_0 := 0,\,
A_0 := \alpha_0.$

\textbf{$\boldsymbol{k+1}$ - шаг:}
\begin{center}
	Найти наибольший корень $\alpha_{k+1} : A_k + \alpha_{k+1} = L_{k+1}\alpha^2_{k+1}$,
\end{center}
\begin{gather*}
A_{k+1} := A_k + \alpha_{k+1},
\end{gather*}
\begin{gather}
y_{k+1} := \frac{\alpha_{k+1}u_k + A_k x_k}{A_{k+1}} \label{eqymir2DL},
\end{gather}
\begin{equation}
\begin{gathered}
\phi_{k+1}(x) = V(x, u_k) + \alpha_{k+1}\psi_{\delta_k}(x, y_{k+1}),\\
u_{k+1} := {\argmin_{x \in Q}}^{\widetilde{\delta}_k}\phi_{k+1}(x), \label{equmir2DL}
\end{gathered}Просуммируем
\end{equation}
\begin{gather}
x_{k+1} := \frac{\alpha_{k+1}u_{k+1} + A_k x_k}{A_{k+1}} \label{eqxmir2DL}.
\end{gather}
Если выполнено условие
\begin{equation}
\begin{gathered}
F_{\delta_k}(x_{k+1}) \leq F_{\delta_k}(y_{k+1}) + \psi_{\delta_k}(x_{k+1}, y_{k+1}) +\\  + \frac{L_{k+1}}{2}\norm{x_{k+1} - y_{k+1}}^2 + \delta_k,
\label{exitLDL}
\end{gathered}
\end{equation}
то
\begin{gather*}
L_{k+2} := \frac{L_{k+1}}{2}
\end{gather*}
и перейти к следующему шагу, иначе
\begin{gather*}
L_{k+1} := 2L_{k+1}
\end{gather*}
и повторить текущий шаг.

\begin{lemma}
	\label{lemma_maxmin_1}
	Пусть для последовательности $\alpha_k$ выполнено
	\begin{align*}
	\alpha_0 = 0,\,\,\,
	A_k = \sum_{i = 0}^{k}\alpha_i,\,\,\,
	A_k = L_{k}\alpha_k^2,\,\,\,
	\end{align*}
	где $L_k \leq 2L$ для любого $k\geq0$ (замечание \ref{remark_maxmin}).
	Тогда для любого $k \geq 1$ верно следующее неравенство
	 \begin{align}
	 \label{lemma_maxmin_1_1}
	 A_k \geq \frac{(k+1)^2}{8L}.
	 \end{align}
\end{lemma}

\begin{proof}\renewcommand{\qedsymbol}{}
	Пусть $k = 1$. Таким образом,
	\begin{equation*}
	\alpha_1 = L_{1}\alpha_1^2
	\end{equation*}
	и
	\begin{equation*}
	A_1 = \alpha_1 = \frac{1}{L_1} \geq \frac{1}{2L}.
	\end{equation*}
	Пусть $k \geq 2$, тогда верны следующие эквивалентные равенства:
	\begin{equation*}
	L_{k+1}\alpha^2_{k+1} = A_{k+1}\Leftrightarrow
	\end{equation*}
	\begin{equation*}
	L_{k+1}\alpha^2_{k+1} = A_{k} + \alpha_{k+1}\Leftrightarrow
	\end{equation*}
	\begin{equation*}
	L_{k+1}\alpha^2_{k+1} - \alpha_{k+1} - A_{k} = 0.
	\end{equation*}
	Решая данное квадратное уравнение будем брать наибольший корень:
	\begin{equation*} 
	\alpha_{k+1} = \frac{1 + \sqrt{\uprule 1 + 4L_{k+1}A_{k}}}{2L_{k+1}}.
	\end{equation*}
	По индукции, пусть неравенство (\ref{lemma_maxmin_1_1}) верно для $k$, тогда:
	\begin{gather*}
	\alpha_{k+1} = \frac{1}{2L_{k+1}} + \sqrt{\frac{1}{4L_{k+1}^2} + \frac{A_{k}}{L_{k+1}}} \geq 
	\frac{1}{2L_{k+1}} + \sqrt{\frac{A_{k}}{L_{k+1}}} \geq \\\geq
	\frac{1}{4L} + \frac{1}{\sqrt{2L}}\frac{k+1}{2\sqrt{2L}} =
	\frac{k+2}{4L}.
	\end{gather*}
Последнее неравенство следует из индукционного предположения. В конечном счете получаем, что
	\begin{equation*}
	\alpha_{k+1} \geq \frac{k+2}{4L}
	\end{equation*}
и
	\begin{equation*}
	A_{k+1} = A_k + \alpha_{k+1} = \frac{(k+1)^2}{8L} + \frac{k+2}{4L} \geq \frac{(k+2)^2}{8L}.
	\end{equation*}
\end{proof}

Докажем основную лемму.

\begin{lemma}
\label{lemma_maxmin_3DL}
	Для любого $x \in Q$ выполнено неравенство
	\begin{equation*}
		A_{k+1} F(x_{k+1}) - A_{k} F(x_{k}) + V(x, u_{k+1}) - V(x, u_{k}) \leq \alpha_{k+1}F(x) + 2\delta_k A_{k+1} + \widetilde{\delta}_k.
	\end{equation*}
\end{lemma}
\begin{proof}\renewcommand{\qedsymbol}{}
	Рассмотрим цепочку неравенств:
	\begin{gather*}
	F(x_{k+1}) \leqarg{(\ref{exitLDL}), (\ref{exitLDLOrig2})} F_{\delta_k}(y_{k+1}) + \psi_{\delta_k}(x_{k+1},y_{k+1})  + \frac{L_{k+1}}{2}\norm{x_{k+1} - y_{k+1}}^2 + 2\delta_k \eqarg{(\ref{eqxmir2DL})} \\=
	F_{\delta_k}(y_{k+1}) + \psi_{\delta_k}(\frac{\alpha_{k+1}u_{k+1} + A_k x_k}{A_{k+1}},y_{k+1}) +\\+ \frac{L_{k+1}}{2}\norm{\frac{\alpha_{k+1}u_{k+1} + A_k x_k}{A_{k+1}} - y_{k+1}}^2 + 2\delta_k\leqarg{Опр. \ref{gen_delta_L_oracle}, (\ref{eqymir2DL})} \\\leq
	F_{\delta_k}(y_{k+1}) + 
	\frac{\alpha_{k+1}}{A_{k+1}}\psi_{\delta_k}(u_{k+1}, y_{k+1}) + \\+
	 \frac{A_k}{A_{k+1}}\psi_{\delta_k}(x_k, y_{k+1})  + \frac{L_{k+1} \alpha^2_{k+1}}{2 A^2_{k+1}}\norm{u_{k+1} - u_k}^2 + 2\delta_k= \\=
	 \frac{A_k}{A_{k+1}}(F_{\delta_k}(y_{k+1}) + \psi_{\delta_k}(x_k, y_{k+1}))
	 + \\+
	 \frac{\alpha_{k+1}}{A_{k+1}}(F_{\delta_k}(y_{k+1}) + 
	 \psi_{\delta_k}(u_{k+1}, y_{k+1}))+ \\
	   + \frac{L_{k+1} \alpha^2_{k+1}}{2 A^2_{k+1}}\norm{u_{k+1} - u_k}^2 + 2\delta_k=_{{\tiny \circled{1}}} \\ =
	 \frac{A_k}{A_{k+1}}(F_{\delta_k}(y_{k+1}) + \psi_{\delta_k}(x_k,y_{k+1}))
	 + \\+
	 \frac{\alpha_{k+1}}{A_{k+1}}(F_{\delta_k}(y_{k+1}) + \psi_{\delta_k}(u_{k+1},y_{k+1})
	 + \frac{1}{2 \alpha_{k+1}}\norm{u_{k+1} - u_k}^2) + 2\delta_k\leq \\\leq
	 \frac{A_k}{A_{k+1}}(F_{\delta_k}(y_{k+1}) + \psi_{\delta_k}(x_k,y_{k+1}))
	 + \\+
	 \frac{\alpha_{k+1}}{A_{k+1}}(F_{\delta_k}(y_{k+1}) + \psi_{\delta_k}(u_{k+1},y_{k+1})
	 + \frac{1}{\alpha_{k+1}}V(u_{k+1}, u_k)) + 2\delta_k\leq_{{\tiny \circled{2}}} \\\leq
	 \frac{A_k}{A_{k+1}} F(x_k) + \\+
	 \frac{\alpha_{k+1}}{A_{k+1}}(F_{\delta_k}(y_{k+1}) + \psi_{\delta_k}(x,y_{k+1})
	 + \frac{1}{\alpha_{k+1}}V(x, u_k) - \frac{1}{\alpha_{k+1}}V(x, u_{k+1}) + \frac{\widetilde{\delta}_k}{\alpha_{k+1}}) + 2\delta_k \leqarg{(\ref{exitLDLOrig})} \\\leq
	 \frac{A_k}{A_{k+1}} F(x_k) + 
	 \frac{\alpha_{k+1}}{A_{k+1}} F(x)
	 + \frac{1}{A_{k+1}}V(x, u_k) - \frac{1}{A_{k+1}}V(x, u_{k+1}) + 2\delta_k + \frac{\widetilde{\delta}_k}{A_{k+1}}
	\end{gather*}
	
	Неравенство {\small \circled{1}} следует из равенства $A_k = L_{k}\alpha^2_k$. Неравенство {\small \circled{2}} следует из леммы \ref{lemma_maxmin_2} с 
	$\psi(x) = \alpha_{k+1} \psi_{\delta_k}(x, y_{k+1})$ и левой части (\ref{exitLDLOrig}).
\end{proof}

\begin{theorem}
	\label{mainTheoremDL}
	Пусть $V(x_*, x_0) \leq R^2$, где $x_0$ -- начальная точка, а $x_*$ -- ближайшая точка минимума к точке $x_0$ в смысле дивергенции Брэгмана. Для предложенного алгоритма выполнено следующее неравенство:
	\begin{equation*}
	F(x_N) - F(x_*) \leq \frac{8LR^2}{(N+1)^2} + \frac{2\sum_{k = 0}^{N-1}\delta_kA_{k+1}}{A_{N}} + \frac{8L\sum_{k=0}^{N-1}\widetilde{\delta}_k}{(N+1)^2}.
	\end{equation*}
\end{theorem}
\begin{proof}\renewcommand{\qedsymbol}{}
	Суммируя неравенство из леммы \ref{lemma_maxmin_3DL} по $k = 0, ..., N - 1$, получим
	\begin{gather*}
		A_{N} F(x_N) - A_{0} F(x_0) + V(x, u_N) - V(x, u_0) \leq (A_N - A_0)F(x) + 2\sum_{k = 0}^{N-1}\delta_kA_{k+1} + \sum_{k=0}^{N-1}\widetilde{\delta}_k \Leftrightarrow\\
		A_{N} F(x_N) + V(x, u_N) - V(x, u_0) \leq A_NF(x) + 2\sum_{k = 0}^{N-1}\delta_kA_{k+1} + \sum_{k=0}^{N-1}\widetilde{\delta}_k.
	\end{gather*}
	Если взять $x = x_*$, то будет верно неравенство
	\begin{gather*}
		A_{N} (F(x_N) - F_*) \leq R^2 + 2\sum_{k = 0}^{N-1}\delta_kA_{k+1} + \sum_{k=0}^{N-1}\widetilde{\delta}_k
	\end{gather*}
	Разделим обе части неравенства на $A_N$ и в конечном счете получим, что
	\begin{gather*}	
		F(x_N) - F_* \leq \frac{R^2}{A_{N}} + \frac{2\sum_{k = 0}^{N-1}\delta_kA_{k+1}}{A_{N}} + \frac{\sum_{k=0}^{N-1}\widetilde{\delta}_k}{A_{N}} \leq_{{\tiny \circled{1}}}\\\leq \frac{8LR^2}{(N+1)^2} + \frac{2\sum_{k = 0}^{N-1}\delta_kA_{k+1}}{A_{N}} + \frac{8L\sum_{k=0}^{N-1}\widetilde{\delta}_k}{(N+1)^2}.
	\end{gather*}
		
	Неравенство {\small \circled{1}} следует из леммы \ref{lemma_maxmin_1}.
	
\end{proof}

\section{Следствия}
\label{sledviya}
\subsection{\it Быстрый градиентный метод}
Предположим, что $F(x)$ -- гладкая выпуклая функция с $L$-липшицевым градиентом в норме $\norm{}$, тогда (см. \cite{devolder2014first})
\begin{gather}
0 \leq F(x) - F(y) - \langle\nabla F(y), x - y \rangle \leq \frac{L}{2}\norm{x - y}^2  \,\,\,\, \forall x,y \in Q.
\end{gather}
Таким образом, получаем, что $\psi_{\delta_k}(x,y) = \langle\nabla F(y), x - y \rangle$, $F_{\delta_k}(y) = F(y)$ и $\delta_k = 0 ,\, \forall k$. Помимо этого будем предполагать, что промежуточную задачу мы можем решать точно, поэтому $\widetilde{\delta}_k = 0$ для любого $k$. Получаем следующую скорость сходимости для быстрого варианта метода (разд. \ref{fastGradMethod} и теорема \ref{mainTheoremDL}):

\begin{equation*}
F(x_N) - F(x_*) \leq \frac{8LR^2}{(N+1)^2}.
\end{equation*}

Данная скорость сходимости является оптимальной с точностью до числового множителя (лучше константы 2 получить нельзя \cite{taylor2015exact}, в то время, как у нас константа 8).

\subsection{\it Сравнение градиентного и быстрого градиентного метода}
Предположим, что у нас для задачи (\ref{mainTask3}) имеется $(\delta, L)$-оракул из \cite{devolder2014first} и будем считать, что промежуточную задачу в смысле определения \ref{solNemirovskiy} мы можем решать на каждом шаге с точностью $\widetilde{\delta}$, тогда (теоремы \ref{mainTheoremDL_G} и \ref{mainTheoremDL}) верны следующие неравенства:
\begin{gather*}
F(\bar{x}_N) - F(x_*) \leq \frac{2LR^2}{N} + 2L\widetilde{\delta} + 2\delta,\\
F(x_N) - F(x_*) \leq \frac{8LR^2}{(N+1)^2}  + \frac{8L\widetilde{\delta}}{N+1} + 2N\delta ,
\end{gather*}
где $\bar{x}_N$ -- точка из теоремы \ref{mainTheoremDL_G}, а $x_N$ -- точка из теоремы \ref{mainTheoremDL}. Мы получаем, что быстрый вариант более устойчив к ошибкам $\widetilde{\delta}$, возникающим при решении промежуточных задач, но при этом он накапливает ошибки $\delta$, которые возникают при вызове $(\delta, L)$-оракула. Стоит отметить, что существует промежуточный градиентный метод \cite{devolder2013intermediate} (для стохастического случая \cite{dvurechensky2016stochastic}), для которого можно получить оценку вида \begin{gather*}\mathcal{O}(1)\frac{LR^2}{N^p} + O(1)N^{1-p}\widetilde{\delta} + \mathcal{O}(1)N^{p-1}\delta,\end{gather*} где $p \in [1,2]$ можно выбирать произвольным образом и за счет этого уменьшать непрерывно влияние шума, но при этом ухудшать оценку скорости сходимости.

\subsection{\it Универсальный метод}

Рассмотрим быструю версию градиентного метода (разд. \ref{fastGradMethod} и теорема \ref{mainTheoremDL}). 

Универсальный метод \cite{nesterov2015universal} позволяет применять концепцию $(\delta,L)$-оракула \cite{devolder2014first},\cite{nesterov2015universal} для решения негладких задач. Будем предполагать, что выполняется условие Гёльдера: существует $\nu\in[0,1]$ такое, что
\begin{equation*}
\norm{\nabla F(x) - \nabla F(y)}_* \leq L_\nu\norm{x - y}^\nu\,\,\,\,\forall x,y \in Q.
\end{equation*}
Тогда (см. \cite{nesterov2015universal})
\begin{gather}
0 \leq F(x) - F(y) - \langle\nabla F(y), x - y \rangle \leq \frac{L(\delta)}{2}\norm{x - y}^2 + \delta \,\,\,\, \forall x,y \in Q,
\end{gather}
где \begin{gather*}L(\delta)=L_\nu\left[\frac{L_\nu}{2\delta}\frac{1-\nu}{1+\nu}\right]^\frac{1-\nu}{1+\nu}\end{gather*} и $\delta > 0$ -- свободный параметр. Получаем, что $\psi_{\delta_k}(x,y) = \langle\nabla F(y), x - y \rangle$, $F_{\delta_k}(y) = F(y)$. Будем предполагать, что промежуточную задачу мы можем решать точно, таким образом, $\widetilde{\delta}_k = 0$ для любого $
k$. Возьмем \begin{gather}\delta_k = \epsilon \frac{\alpha_{k+1}}{4A_{k+1}} \,\,\,\,\forall k,\label{delta_k_univ}\end{gather} где $\epsilon$ -- необходимая точность решения по функции.

\change{Из теоремы \ref{mainTheoremDL} с нашими предположениями мы имеем следующую скорость сходимости:
	\begin{equation}
	f(x_N) - f(x_*) \leq \frac{R^2}{A_N} + \frac{\epsilon}{2}.
	\end{equation} 
	Как и в статье \cite{nesterov2015universal} мы можем показать, что верно следущее неравенство:
$$A_N \geq \frac{N^\frac{1+3\nu}{1+\nu}\epsilon^\frac{1-\nu}{1+\nu}}{2^\frac{2+4\nu}{1+\nu}L_\nu^\frac{2}{1+\nu}}.$$
Отсюда получаем, что
$$N \leq \inf_{\nu\in[0,1]}\left[2^\frac{3+5\nu}{1+3\nu}\left(\frac{L_\nu R^{1+\nu}}{\epsilon}\right)^\frac{2}{1+3\nu}\right].$$}

Данная оценка является оптимальной с точностью до числового множителя \cite{guzman2015lower}.

\subsection{\it Метод условного градиента}

На практике часто вспомогательная задача (\ref{equmir2DL}) не может быть решена за разумное время \cite{ben-tal2015lectures}, \cite{nesterov2015complexity}. В работе \cite{jaggi2013revisiting} показывается, что  метод условного градиента (Франк--Вульфа) \cite{ben-tal2015lectures}, \cite{jaggi2013revisiting}, \cite{harchaoui2015conditional}  может быть очень эффективен для определенного класса задач. Поэтому вместо $\phi_{k+1}(x) = V(x, u_k) + \alpha_{k+1}\psi_{\delta_k}(x, y_{k+1})$ в (\ref{equmir2DL}) используют $\widetilde{\phi}_{k+1}(x) = \alpha_{k+1}\psi_{\delta_k}(x, y_{k+1})$. Рассмотрим данную замену с точки зрения ошибки $\widetilde{\delta}_k$. Далее будем предполагать, что $F(x)$ -- гладкая выпуклая функция с $L$-липшицевым градиентом в норме $\norm{}$ и $V(x,y) \leq R_Q^2$ для любого $x,y \in Q$. И пусть $u_{k+1} = \left(\argmin^{\widetilde{\delta}_k}_{x \in Q}\phi_{k+1}(x) \myeq \argmin_{x \in Q} \widetilde{\phi}_{k+1}(x)\right)$, тогда
\begin{gather*}
\exists h \in \partial\phi_{k+1}(u_{k+1}), \exists g \in \partial\widetilde{\phi}_{k+1}(u_{k+1}),\,\,\, \langle h, x - u_{k+1}  \rangle =\\= \langle g, x - u_{k+1}  \rangle + \langle \nabla_{u_{k+1}} V(u_{k+1}, u_k), x - u_{k+1}\rangle \geq \langle \nabla_{u_{k+1}} V(u_{k+1}, u_k), x - u_{k+1}\rangle = \\ =
-V(u_{k+1},u_{k}) - V(x,u_{k+1}) + V(x,u_{k}) \geq -2R^2_Q.
\end{gather*}

В алгоритме будем предполагать, что $\widetilde{\delta}_k = 2R^2_Q$ для любого $k$. Остальное аналогично гладкому случаю с $L$-липшицевым градиентом в норме $\norm{}$. Тогда получаем следующую скорость сходимости для быстрого варианта метода (разд. \ref{fastGradMethod} и теорема \ref{mainTheoremDL}):
\begin{equation*}
F(x_N) - F(x_*) \leq \frac{8LR^2}{(N+1)^2} + \frac{16LR^2_Q}{N+1}.
\end{equation*}

Данная оценка с точностью до числового множителя не может быть улучшена для приведенного метода \cite{polyak1983introduction}.

\subsection{\it Композитная оптимизация}
Рассмотрим задачу композитной оптимизации \cite{nesterov2013gradient}:
\begin{align}
\label{composite}
F(x) \myeq f(x) + h(x) \rightarrow \min_{x \in Q},
\end{align}
где $f(x)$ -- гладкая выпуклая функция с $L$-липшицевым градиентом в норме $\norm{}$ и $h(x)$ -- выпуклая функция (в общем случае негладкая). Для данной задачи верно следующее неравенство
\begin{gather}
0 \leq F(x) - F(y) - \langle\nabla f(y), x - y \rangle - h(x) + h(y) \leq \frac{L}{2}\norm{x - y}^2  \,\,\,\, \forall x,y \in Q.
\end{gather}
Таким образом, мы можем взять $\psi_{\delta_k}(x,y) = \langle\nabla f(y), x - y \rangle + h(x) - h(y)$, $F_{\delta_k}(y) = F(y)$ и $\delta_k = 0$ для любого $k$. Получается, что для данной задачи обычный и ускоренный варианты метода будут работать без изменений. Стоит отметить, что часть сложности задачи мы таким образом переносим в (\ref{equmir2DL_G}) или (\ref{equmir2DL}). Если в гладком случае в вспомогательной задаче стоит функция вида $V(x, u_k) + \alpha_{k+1}\langle\nabla f(y_{k+1}), x - y_{k+1} \rangle$, то в данной задаче (\ref{composite}) добавляется слагаемое $h(x)$, и в конечном счете нужно решать более сложную задачу на каждом шаге вида $V(x, u_k) + \alpha_{k+1}\left(\langle\nabla f(y_{k+1}), x - y_{k+1} \rangle + h(x) - h(y_{k+1})\right)$.

\subsection{\it Прокс-метод}
Рассмотрим задачу:
\begin{align}
\label{main_prox_method}
F(x) \rightarrow \min_{x \in Q},
\end{align}
где $F(x)$ -- в общем случае негладкая выпуклая функция. Можем в описанном выше подходе выбрать $\psi_{\delta_k}(x,y) = F(x) - F(y)$, $F_{\delta_k}(y) = F(y)$ и $\delta_k = 0$ для любого $k$. Условие (\ref{exitLDLOrig}) будет выполнятся при любом выборе $L \geq 0$. Методы из разд. \ref{gradMethod} и разд. \ref{fastGradMethod} являются, вообще говоря, адаптивными в том смысле, что "локальная"\,константа Липшица градиента $L_k$ подбирается во время работы метода. Зафиксируем произвольную константу $L \geq 0$, возьмем все $L_k$ равными $L$, не подбирая их во внутренним цикле. При этом довольно легко показать, что метод и все оценки останутся без изменений. Тогда промежуточный шаг для метода из разд. \ref{gradMethod} будет выглядеть следующим образом:
\begin{align}
x_{k+1} := {\argmin_{x \in Q}}^{\widetilde{\delta}_k}\left[LV(x, x_k) + F(x)\right].
\label{prox_help}
\end{align}
Данный метод называется проксимальным методом \cite{polyak1983introduction}, \cite{parikh2014proximal}, он может быть эффективным в ряде задач \cite{lin2015universal}. Для негладкой функции алгоритм из разд. \ref{fastGradMethod} сходится по оценкам теоремы \ref{gradMethod}, что противоречит нижним оценкам для негладких функций \cite{nesterov2010introductory}, \cite{nemirovskiy1979slognost}. Однако задача (\ref{prox_help}), вообще говоря, может быть решена только приближено \cite{rakhlin2012making}, \cite{juditsky2011first}. Более детальный анализ показывает \cite{gasnikov2017universal}, что общее число обращений к оракулу за субградиентом функции $f(x)$ не противоречит нижним оценкам \cite{nesterov2010introductory}, \cite{nemirovskiy1979slognost} для негладких задач, и более того, соответствует, им.

\subsection{\it Суперпозиция функций}
Рассмотрим следующую задачу (см. \cite{nemirovski1995information}, \cite{lan2015bundle},\cite{nemirovskiy1985optimalnie}):
\begin{align}
F(x) \myeq f(f_1(x), \dots, f_m(x)) \rightarrow \min_{x \in Q},
\end{align}
где $f_k(x)$ -- гладкая выпуклая функция с $L_k$-липшицевым градиентом в норме $\norm{}$ для любого $k$. Функция $f(x)$ является $M$-липшицевой выпуклой функцией относительно $L_1$-нормы, неубывающей по каждому из своих аргументов. Как следствие (см. \cite{boyd2004convex}, \cite{lan2015bundle}), функция $F(x)$ так же является выпуклой функцией и выполняется следующее неравенство \cite{lan2015bundle}:
\begin{gather*}
0 \leq F(x) - f(f_1(y) + \langle\nabla f_1(y), x - y \rangle, \dots, f_m(y)+\langle\nabla f_m(y), x - y \rangle) \leq\\\leq M\frac{\sum_{i=1}^{m}L_i}{2}\norm{x - y}^2  \,\,\,\, \forall x,y \in Q.
\end{gather*}
И верно
\begin{gather*}
0 \leq F(x) - F(y) - f(f_1(y) + \langle\nabla f_1(y), x - y \rangle, \dots, f_m(y)+\langle\nabla f_m(y), x - y \rangle) + F(y) \leq\\\leq M\frac{\sum_{i=1}^{m}L_i}{2}\norm{x - y}^2  \,\,\,\, \forall x,y \in Q.
\end{gather*}
Мы можем взять $\psi_{\delta_k}(x,y) = f(f_1(y) + \langle\nabla f_1(y), x - y \rangle, \dots, f_m(y)+\langle\nabla f_m(y), x - y \rangle) - F(y)$, $F_{\delta_k}(y) = F(y)$ и $\delta_k = 0$ для любого $k$. Как и в задаче (\ref{composite}) оценки на скорость сходимости сохраняются, но при этом вспомогательные задачи (\ref{equmir2DL_G}) и (\ref{equmir2DL}) могут сильно усложниться. Данная задача может включать обширное количество частных случаев \cite{lan2015bundle}, \cite{nemirovski1995information}: гладкая оптимизация, негладкая оптимизация, минимаксная задача \cite{nesterov2010introductory}, композитная оптимизация, задача с регуляризацией.

\subsection{\it Дополнительные примеры}
Рассмотрим без подробного объяснения дополнительные примеры постановок задач, в которых может быть актуальной концепция модели функции из разд. \ref{gradMethod}.
\begin{enumerate}
\item Рассматривается следующая минмин задача \cite{gasnikoveffectivnie}:
\begin{align}
f(x) \myeq \min_{y \in Q}F(y,x) \rightarrow \min_{x \in \mathds{R}^n}.
\label{main_minmin}
\end{align}
Пусть $F(y,x)$ -- гладкая и выполнено условие:
\begin{gather*}
\norm{\nabla F(y',x') -\nabla F(y,x)}_2 \leq L \norm{(y',x') -(y,x)}_2 \,\,\,\, \forall y,y' \in Q, \,\,\,\,\forall x,x' \in \mathds{R}^n.
\end{gather*}
Тогда (из \cite{gasnikov2016stochatichiskie}), если можно найти такую $\widetilde{y}_\delta(x) \in Q$, что верно неравенство
\begin{gather*}
\langle\nabla_yF(\widetilde{y}_\delta(x), x), y - \widetilde{y}_\delta(x)\rangle \geq -\delta \, \,\,\, \forall y \in Q,
\end{gather*}
то 
\begin{gather*}
F(\widetilde{y}_\delta(x), x) - f(x) \leq \delta, \norm{\nabla f(x') -\nabla f(x)}_2 \leq L \norm{x' -x}_2,
\end{gather*}
и
\begin{gather*}
\left(F_{\delta}(x)= F(\widetilde{y}_\delta(x), x) - 2\delta, \psi_{\delta}(z,x) = \langle\nabla_yF(\widetilde{y}_\delta(x), x), z - x\rangle\right)
\end{gather*}
будет $(6\delta, 2L)$-моделью для функции $f(x)$ в точке $x$.

Таким образом, мы получаем $(6\delta, 2L)$-модель, которая может быть использована для решения (\ref{main_minmin}).
\item Рассматривается следующая задача поиска седловой точки \cite{gasnikoveffectivnie}:
\begin{align}
f(x) \myeq \max_{y \in Q}\left[\langle x, b - Ay\rangle - \phi(y)\right] \rightarrow \min_{x \in \mathds{R}^n},
\label{main_saddle}
\end{align}
где $\phi(y)$ является $\mu$-сильно выпуклой относительно $p$-нормы, $1\leq p\leq2$. Тогда из \cite{devolder2014first} функция $f(x)$ -- гладкая функция с константой Липшица градиента в 2-норме
\begin{gather*}
L = \frac{1}{\mu}\max_{\norm{y}_p\leq1}\norm{Ay}_2^2.
\end{gather*}
Если $y_\delta(x)$ -- решение вспомогательной задачи максимизации с точностью по функции $\delta$, то пара
\begin{gather*}
\left(F_{\delta}(x)=\langle x, b - Ay_\delta(x)\rangle - \phi(y_\delta(x)), \psi_{\delta}(z,x) =\langle b - Ay_\delta(x), z - x\rangle\right)
\end{gather*}
будет $(\delta, 2L)$-моделью для функции $f(x)$ в точке $x$.

\item Рассматривается следующая функция (следует сравнить c (\ref{prox_help})):
\begin{align}
f(x) \myeq \min_{y \in Q}\underbrace{\left\{\phi(y) + \frac{L}{2}\norm{y - x}^2_2\right\}}_{\Lambda(x,y)}.
\label{prox_ex}
\end{align}
Пусть $\phi(y)$ -- выпуклая функция и 
\begin{gather*}
\max_{y \in Q}\left\{\Lambda(x,y(x)) - \Lambda(x,y) + \frac{L}{2}\norm{y - y(x)}^2_2\right\} \leq \delta
\end{gather*}
Тогда из \cite{devolder2014first} верно, что
\begin{gather*}
\left(F_{\delta}(x)=\phi(y(x)) + \frac{L}{2}\norm{y(x) - x}^2_2 - \delta,\psi_{\delta}(z,x) =\langle L(x - y(x)), z - x\rangle\right)
\end{gather*}
будет $(\delta, L)$-моделью для функции $f(x)$ в точке $x$.
\end{enumerate}

\section{Заключение}
В данной работе представлен градиентный и быстрый градиентный метод для $(\delta,L)$-модели. В работе приведены алгоритмы методов и оценки скоростей сходимости. В разд. \ref{sledviya} показано, что данные методы могут быть довольно сильным инструментом для решения большого класса задач. Стоит отметить, что список задач из разд. \ref{sledviya}, который покрывает данная работа, не исчерпывает полный потенциал концепции. Мы верим, что данный подход может быть применим во многих других проблемах, включая стохастическую, покомпонентную, безградиентную оптимизацию \cite{tyurin2017mirror}. Также можно показать, что предложенные в статье алгоритмы являются прямо-двойственными \cite{anikin2017dual}. Детали планируется изложить в будущих исследованиях.

Хотим поблагодарить Павла Двуреченского за указание на ряд литературных источников.

\printbibliography

\begin{center}
	\Large \it Приложение
\end{center}

В данной работе ключевым образом использовался тот факт, что мы решаем вспомогательную задачу с точностью $\widetilde{\delta}$, используя концепцию из определения \ref{solNemirovskiy}. Было показано, что из $\widetilde{\delta}$-решения в смысле определения \ref{solNemirovskiy} следует $\widetilde{\delta}$-решение по функции. Обратное в общем случае не всегда верно, но мы попробуем представить довольно общие примеры, когда это будет выполнимо. Тривиальный случай, когда $\widetilde{\delta} = 0$. Тогда из критерия оптимальности первого порядка будет следовать, что данные два определения $\widetilde{\delta}$-решения будут эквивалентны.

Предположим, что решается следующая задача:
\begin{align}
\label{helpTask}
\alpha_k\psi(x) + V(x,x_k) \rightarrow \min_{x \in Q},
\end{align}
где $\psi(x)$ -- выпуклая функция и $V(x,x_k)$ -- сильно-выпуклая функция с константой сильной выпуклости равной 1. Вспомогательная задача в итерациях методов оптимизации довольно часто имеет такой вид. Конечно, существует случаи, когда данную задачу можно решить аналитически, например, когда в основной задаче решается гладкая задача оптимизации без ограничений с евклидовой прокс-структурой ($V(x,y) = \frac{1}{2}\norm{x-y}_2^2$). В случаях, когда задача (\ref{helpTask}) может быть решена только численно, можно использовать различные подходы в зависимости от задачи.

Рассмотрим случай, когда $$\psi(x) + V(x,x_k) = \sum_{i=1}^{n}\left[\psi_i(x_i) + V_i(x_i)\right].$$ При данном условии задача в (\ref{helpTask}) является сепарабельной. Получается, что достаточно решить $n$ одномерных задач, каждую из которых можно решать методом деления отрезка пополам \cite{vasiliev2017methods} за время $\mathcal{O}(\ln\left(\frac{1}{\epsilon}\right))$, где $\epsilon$ -- точность по функции.

Можно использовать два подхода, если дополнительно предположить, что $\psi(x)$ с $L$-липшицевым градиентом в норме $\norm{}$. Если $V(x,x_k)$ с $L$-липшицевым градиентом в норме $\norm{}$, то задачу можно решать за линейное время $\mathcal{O}(\ln\left(\frac{1}{\epsilon}\right))$ \cite{nesterov2010introductory}. Если $V(x,x_k)$ не имеет $L$-липшицевого градиента в норме $\norm{}$, можно в задаче (\ref{helpTask}) смотреть на $V(x,x_k)$, как на композит. При этом, чтобы получить линейную скорость сходимости, можно воспользоваться техникой рестартов \cite{dvurechensky2016stochastic},\cite{juditsky2014deterministic}.

\end{document}